\documentclass{article}
\usepackage{graphicx} 
\usepackage[T1]{fontenc}
\usepackage{geometry}
\usepackage{amsfonts}
\usepackage{amsmath}
\usepackage{amsthm}
\usepackage{amssymb}
\usepackage{eurosym}
\usepackage{graphicx}
\usepackage{epsfig}
\usepackage{hyperref}
\usepackage{dsfont}
\usepackage{color}
\usepackage{xcolor}
\usepackage{appendix}
\usepackage{authblk}

\allowdisplaybreaks

\usepackage{mathtools}
\usepackage{verbatim}
\usepackage{ifthen}
\hypersetup{urlcolor=blue, citecolor=red}
\usepackage{hyperref}
\newtheorem{theorem}{Theorem}[section]
\newtheorem{corollary}{Corollary}

\newtheorem{lemma}[theorem]{Lemma}

\theoremstyle{definition}
\newtheorem{definition}[theorem]{Definition}
\newtheorem{remark}{Remark}

\newcommand{\Rd}{\mathbb{R}^d}
\newcommand{\R}{\mathbb{R}}

\title{A Linear Test for Global Nonlinear Controllability}
\author{Karthik Elamvazhuthi}
\affil{University of California, Riverside}

\begin{document}
	\date{}
	\maketitle

	\begin{abstract}
		It is known that if a nonlinear control affine system without drift is bracket generating, then its associated sub-Laplacian is invertible under some conditions on the domain. In this note, we investigate the converse. We show how invertibility of the sub-Laplacian operator implies a weaker form of controllability, where the reachable sets of a neighborhood of a point have full measure. From a computational point of view, one can then use the spectral gap of the (infinite-dimensional) self-adjoint operator to define a notion of degree of controllability.
		
		An essential tool to establish the converse result is to use the relation between invertibility of the sub-Laplacian to the the controllability of the corresponding continuity equation using possibly non-smooth controls. Then using Ambrosio-Gigli-Savare's superposition principle from optimal transport theory we relate it to controllability properties of the control system. While the proof can be considered of the Perron-Frobenius type, we also provide a second dual Koopman point of view.
	\end{abstract}
	
	\section{Introduction}

	Suppose we have a $m\leq d$ smooth globally Lipschitz vector fields $g_i \in C^{\infty}(\Rd;\Rd)$, for $i =1,...,m$. A standard problem in control theory is if given $x, y \in \Rd $, does there exist a control $u(t)$ such that the system of ordinary differential equations (ODE)
	\begin{equation}
		\dot{\gamma}(t) = \sum^m_{i=1}u_i(t) g_i(\gamma(t)) 
		\label{eq:ctrsys}
	\end{equation}
	\[\gamma(0) = x,~~~\gamma(1)= y \]
	By now there is an extremely rich body of work on this topic \cite{agrachev2013control,isidori1985nonlinear}. In this brief note, we consider the relationship between the control systems of the type \eqref{eq:ctrsys} with a related class of partial differential operators. For this, we define
	associated first order partial differential operators that acts on smooth functions $h(x)$ through the operation,
	\begin{align}
		X_i h &= \sum\limits_{j=1}^d g_{i}^{j}(x) \ \frac{\partial}{\partial x_j}h ,   \nonumber 
	\end{align}
	and its formal adjoint
	\[X_i^*h = \sum_{j=1}^d -\frac{\partial}{\partial x_j} (g_i^j(x) h). ~~~\]
	
	Given these two operators, we can introduce the {\it sub-Laplacian operator},
	\begin{equation}
		\label{eq:nonlap}
		\Delta_H h = -\sum_{i=1}^{m}X_i^*X_ih,  
	\end{equation}
	The relationship between this operator and the control system \eqref{eq:ctrsys} is well understood in the partial differential equation (PDE) community. For example, if the system \eqref{eq:ctrsys} is bracket generating of order $r$\footnote{The system \eqref{eq:ctrsys} is said to be bracket generating of order $r$ if the span of repeated Lie brackets of the the vector fields $\{g_1,...,g_m\}$, of up to order $r$ has full rank at all points.} is a sufficient condition for  {\it hypoellipticity property} \cite{bramanti2022hormander,rothschild1976hypoelliptic} when the vector fields are smooth,  as shown by H{\"o}rmander \cite{hormander1967hypoelliptic}.

	However, hypoellipticity is not a computationally testable hypothesis. Therefore, we are interested in an alternative property of the sub-Laplacian that can be used to infer controllability of the system \eqref{eq:ctrsys}. For example, invertibility of the operator is one such criterion. In this regard the following is known when the system \eqref{eq:ctrsys} is considered on an open bounded set $\Omega$.
	
	\begin{theorem}
		\label{thm:ctrbinv}
		\textbf{(Bracket Generation implies invertibility of the sub-Laplacian)}
		Suppose the system \eqref{eq:ctrsys} is bracket generating of order $r$. Let $\Omega$ be a $\epsilon-\delta$ subdomain of $\Rd$. then $\Delta_H$ is invertible on $L^{\perp}_2(\Omega)$, the subset of functions in $L_2(\Omega)$ that integrate to $0$ on $\Omega$. 
	\end{theorem}

	For the notion of $\epsilon-\delta$ domain and examples, we refer the reader to \cite{garofalo1998lipschitz}. When $\Omega$ is compact manifold without boundary this result is well known, but hard to find a specific source as the results follow very similarly as the case of the classical Riemannian Laplacian operator $\Delta$. Particularly, the local regularity of the solutions due to subellipticity of the operator \eqref{eq:nonlap} can be used to conclude global regularity, which implies the domain of the operator $\Delta_H$ is compact, which implies discreteness of the spectrum and hence invertibilty. The above result for domains with boundary can be found in \cite{nhieu2001neumann}, stated for the case when vector fields $\{g_i \}$ are left-invariant on a Lie group, homogeneous with respect to a dilation and the Lie algebra generated is nilpotent. Invertibility also holds for more general vector fields, as shown in \cite{elamvazhuthi2024denoising}, owing to the results of \cite{garofalo1998lipschitz}. This has also be observed in \cite{khesin2009nonholonomic} with only a partial proof. These results still hold if the system is controllable but not bracket generating as long as the control metric defined by system \eqref{eq:ctrsys} is complete and continuous with respect to the Euclidean distance. 
	
	Our goal in this note is to consider a converse version of the above statement. There are two views one can take on why the converse could be true. The Perron-Frobenius view and the Koopman view \cite{lasota2013chaos}. These approaches have been fruitful in using linear finite-dimensional methods for control of finite dimensional nonlinear systems \cite{vaidya2008lyapunov,mauroy2020koopman}. 
	\subsection{Perron-Frobenius view}
	
	In the Perron-Frobenius view, one looks at how a dynamical system pushes forward an ensemble of particles, represented by probability densities, under it's flow. In the context of control systems, the question translates to how one can use control to manipulate a family of initial conditions simultaneously.
	
	The invertibility of the operator $\Delta_H$ has been used in \cite{khesin2009nonholonomic} to adapt Moser's theorem for transporting densities under the action of the control system \eqref{eq:ctrsys}, by looking at the controllability of the corresponding continuity equation
	\begin{equation}
		\label{eq:lio}
		\frac{\partial \rho}{\partial t} + \nabla_x \cdot (\sum_{i}^m u_i(t,x)g_i(x) \rho) =0,
	\end{equation}
	which describes how an ensemble of initial conditions $x(0) $, distributed according to $\rho_0$, evolve under the action of the controls.
	
	While \cite{khesin2009nonholonomic} starts with the assumption of bracket generating property of \eqref{eq:ctrsys}, this assumption is not required, and one can instead just assume that the operator $\Delta_H$ is invertible, as stated in the form below. The proof then works the same way, except one just has to check the \eqref{eq:lio} is solved in a weak sense. See Section \ref{sec:PF}. One particularly has the following.
	
	\begin{theorem}
		\label{thm:ctycontrol}
		\textbf{(\cite{khesin2009nonholonomic} Invertibility of $\Delta_H$ implies controllability of continuity equation)} Suppose $\Delta_H$ has a bounded inverse on $L^2_{\perp}(\Omega)$. Let $\rho_0, \rho_1 \in L^2(\Omega)$ be functions such that integrate to $1$ over $\Omega$ and are uniformly bounded from below on $\Omega$ by a positive constant $c>0$. Then there exists control law $u :(0,1) \times \Rd \rightarrow \Rd$ such that $\mu_t$ is a solution to the continuity equation \eqref{eq:lio} where $\mu_t$ is given by 
		\begin{equation}
			\mu_t = \rho_0 + t(\rho_1 - \rho_0) \hspace{5mm} t\in [0,T]
		\end{equation}
		Moreover, a choice of controls is the following,
		\begin{equation}
			\label{eq:track}
			u_i(t,x) = \frac{X_i f(x)}{\rho(t,x)}
		\end{equation}
		for almost every $(t,x)\in (0,1) \times \Omega $
		where 
		\[f = \Delta_H^{-1}(\rho_1 - \rho_0)\]
	\end{theorem}
	The validity of this theorem can be seen by expressing equation \eqref{eq:lio} in the form 
	\begin{equation}
		\label{eq:lionew}
		\frac{\partial \rho}{\partial t} +  \sum_{i}^m X^*_i(u_i(t,x)\rho) =0.
	\end{equation}
	Then the the result can be seen formally by plugging in the control \eqref{eq:track} above.
	In fact, invertibility of $\Delta_H$ can be used to not just control between probability densities, but even exactly track trajectories of probability densities that are sufficiently regular \cite{elamvazhuthi2024denoising} (see also remark made in \cite{arguillere2017sub} for the case without boundary). The idea is that if $\rho(t,x)$ is a trajectory of probability densities, then the controls \\
	\[u_i(t,x) =\frac{X_i (\Delta_H^{-1} \partial_t\rho(t,x)) }{\rho(t,x)} \]
	achieves exact tracking for solution $\rho(t,x)$ of \eqref{eq:lio}, as long as the system \eqref{eq:ctrsys} is controllable.  This should be very surprising, as controllability at the ODE level \eqref{eq:ctrsys} from one point to another implies controllability of trajectories of \eqref{eq:lio} along probability densities, provided the densities are sufficiently regular.
	
	Due to scant results on the boundary regularity properties of solutions of $\Delta_Hu = f$ for function $f$, the flow generated by \eqref{eq:track} is not necessarily well defined (an incorrect statement made to the contrary in \cite{khesin2009nonholonomic} that the flow is a diffeomorphism). Therefore, a drawback is that the control constructed this way, even when \eqref{eq:ctrsys} is known to be controllable, is not guaranteed to be smooth (unless $\Omega$ is a compact manifold without boundary). Hence, the corresponding continuity equation \eqref{eq:lio} and ODE \eqref{eq:ctrsys} can develop non-unique solutions under the action of this control law, and relating solutions of \eqref{eq:ctrsys} and \eqref{eq:lio} doesn't follow from classical results on flows of ODEs with Lipschitz right hand sides. This is especially an issue when trying to prove a result in the converse direction since one cannot even conclude local regularity of the solutions in the interior of the domain from the invertibility of the operator $\Delta_H$.
	
	Despite this hurdle of non-smoothness, one can still establish a correspondence between the two, thanks to a superposition principle proved in \cite{ambrosio2005gradient}, as we show in this note. The goal is to prove the following result.

	\begin{theorem}
		\label{thm:invctrb}
		\textbf{(Main result: Invertibility implies approximate controllability to balls)}
		Suppose $\Delta_H$ is invertible on $L^2_{\perp}(\Omega)$.  Then the set of states ${\rm Reach}^{-1} (B_R(y)) \subseteq \Omega$ that can reach $B_R(y)$, the open ball around $y$, along trajectories of \eqref{eq:ctrsys} has full Lebesgue measure in $\Omega$, for every $R>0$ and every $ y \in \Omega$.
	\end{theorem}
	
	The result can be stated in an alternative way using the spectral gap $\lambda$ of the operator $\Delta_H$. This is attractive from a computational point of view, as one arrives at a (infinite-dimensional) convex test for nonlinear controllability. Here $WH^{1}_{\Omega}$ is the Horizontal Sobolev space defined in Section \ref{sec:not}.
	
	\begin{corollary}
		\textbf{(Spectral gap as a degree of controllability)}
		Let $\lambda >0 $. Suppose $\Delta_H$ satisfies the Poincar\'e inequality
		\begin{equation} 
			\label{eq:poin}
			\int_{\Omega} |f(x)|^2dx \leq \frac{1}{\lambda} \int_{\Omega}\sum_{i=1}^m |X_i f (x)|^2dx 
		\end{equation}
		for all $f \in WH^1_{\perp}(\Omega)$\footnote{This is the {\it Horizontal} Sobolev space of zero mean square integrable functions that are weakly differentiable along the $g_i$. See section \ref{sec:not}}. Then the set of states ${\rm Reach}^{-1} (B_R(y)) \subseteq \Omega$ that can reach $B_R(y)$ along trajectories of \eqref{eq:ctrsys} has full Lebesgue measure in $\Omega$, for every $R>0$ and every $y \in \Omega$.
		
	\end{corollary}
	
	An advantage of this formulation is that the value $\lambda$ gives a notion of degree of controllability. The spectrum of the operator satisfies ${\rm spec}(-\Delta_H) \subseteq \{ 0 \} \cup [\lambda, \infty)$ (the restriction of the operator on $L^2_{\perp}(\Omega)$ has its spectrum contained in $[\lambda, \infty$). If $\rho_0$ and $\rho_1$ are probability densities are positive everywhere on $\Omega$ and essentially bounded on $\Omega$. Let $f = \rho_1 - \rho_0$. Then from Lax-Milgram theorem \cite{clement2016lax} we have the bounds
	\[  \int_{\Omega}\sum_{i=1}^m |X_i u(x)|^2  \leq  \frac{1}{\lambda}  \|f\|^2_2.\]
	
	Therefore, this gives us a bound for the feedback control \eqref{eq:track}. Smaller $\lambda>0$ implies the operator $\Delta_H$ is close to being non-invertible and hence resulting in a higher difficulty in controlling the system \eqref{eq:lio}, and therefore approximately controlling \eqref{eq:ctrsys}. One can think of the $\Delta_H$ as a nonlinear analogue of the controllability grammian for linear time invariant control systems \cite{hespanha2018linear}.
	
	\subsection{Koopman View}
	
	Another view that one can take on this matter is using the dual Koopman point of view by looking at how functions evaluated along the flow evolve. Suppose $A$ is an invariant set such that $B:=\Omega-A$ is also invariant along the flows along each of the vector fields $g_i$, then formally
	\begin{eqnarray}
		\label{eq:dertoint}
		X_i \xi_A  = \lim_{t \rightarrow 0} \frac{\xi_A (e^{t g_i}(x))- \xi_A(x)}{t} = 0  \\
		X_i \xi_B  = \lim_{t \rightarrow 0} \frac{\xi_B (e^{t g_i}(x))- \xi_B(x)}{t} = 0 \nonumber 
	\end{eqnarray}
	for all $x \in \Omega$, where $\xi_A$ and $\xi_B$ represent the characteristic functions of $A$, and $B$, respectively and $e^{t g_i}$represents the flow map along $g_i$.
	
	Therefore,  $X_i(\frac{1}{\int_A dx}\xi_A-(\frac{1}{\int_B dx}\xi_B)=0$ for each $i = 1,..,m$. This implies $\Delta_H$ is not injective on the set of square integrable functions that integrate to $0$. To formalize this argument, one needs to ensure that  $\xi_A $ and $\xi_B$ are weakly differentiable, since they are clearly not differentiable in the classical sense, and that the first and second equality in the relation \eqref{eq:dertoint} hold true rigorously. This is done in Section \ref{sec:koop}.
	
	Lastly, one would hope that this approximate notion of controllability automatically implies exact controllability of \eqref{eq:ctrsys}. However, in general, the reachable sets of \eqref{eq:ctrsys} could be dense in $\Omega$ without the system being (exactly) controllable. For example, the irrational winding on the torus. See \cite{agrachev2013control}. However, this can possibly be avoided for bilinear systems, and left invariant systems on Lie groups \cite{cannarsa2021approximately}.
	
	\section{Notation and Background}
	\label{sec:not}
	Here on, we will only require that the vector fields $g_i \in C^2(\bar{\Omega};\Rd)$. In this section, we define some notation and some results that will be used to prove Theorem \ref{thm:invctrb}.
	Let $B_R(y)$ denote the open ball of around a point $x \in \Omega$. The space $L^2(\Omega)$ is the set of square integrable functions. We will say that function $f \in L^2(\Omega)$ has weak derivative $u$ along the vector field $g_i$, denoted by $X_if \in L^2(\Omega)$, if
	
	\begin{equation}
		\label{eq:wkderdef}
		\int_{\Omega}u(x) \phi(x)dx = \int_{\Omega}f(x)X^*_i\phi(x)dx
	\end{equation}
	
	for all $\phi \in C^{\infty}_0(\Omega)$.

	We define the space $WH^1(\Omega) = \{ u \in L^2(\Omega); X_i u \in L^2(\Omega) \}$ equipped with the Sobolev norm $\|f\|_{WH^1(\Omega)} = \|f\|_2 + \sum_{i=1}^m\|X_if\|_2$. Additionally, $WH^1_{\perp} : = L^2_{\perp}(\Omega) \cap WH^1(\Omega)$. We will need the form $\sigma: WH^1(\Omega) \times WH^1(\Omega ) \rightarrow \mathbb{R}$. Associated with this form, we define the operator $\Delta_H : \mathcal{D}(\Delta_H) \rightarrow L_2(\Omega)$ with domain
	\[D(\Delta_H): \{ u \in \mathcal{D}(\sigma);\exists f \in L^2(\Omega) ~s.t.~\sigma(u,\phi) = \langle f,\phi \rangle_2 ~ \forall \phi \in WH^1(\Omega)\}\]
	with the operator defined by 
	\[\Delta_H u = f\]
	if 
	\begin{equation}
		\sigma(u,\phi) = \langle f,\phi\rangle_2
		~~\forall \phi \in WH^1(\Omega)
	\end{equation}
	
	We will also need the notion of invertibility of an operator. We will say that the operator $\Delta_H$ is invertible on $L^2_{\perp}(\Omega)$, if there exists an operator $\Delta_H^{-1} : L^2_{\perp} (\Omega)\rightarrow L^2_{\perp}(\Omega)$ such that $\Delta_H \Delta_H^{-1}x = x$ for all $x \in L^2_{\perp}$ and $\Delta^{-1}_H$ is a bounded linear operator. Sometimes, instead of invertibility, we will be happy with injectivity of the operator $\Delta_H$. In general, injectivity is weaker than invertibility for operators on infinite dimensional spaces. For self-adjoint operators, this is the case when the range of the operator is dense without being closed.
	
	An immediate consequence of these definitions is the following result on basic properties of the operator $\Delta_H$.

	\begin{lemma}
		The space $WH^1(\Omega)$ is complete under the norm $\|\cdot\|_{WH^1(\Omega)}$, and
		$\sigma(u,u) \geq 0$
		for all $u \in WH^1(\Omega)$. Hence, the form $\sigma$ is closed and semibounded.  As a consequence, the operator $\Delta_{H}$ is self-adjoint. 
	\end{lemma}
	The completeness of property of the space $WH^1(\Omega)$ can be found in \cite{bramanti2022hormander}. The rest is a classical correspondence between closedness and semiboundedness of a form and self-adjointness of the corresponding operator.

	We will relate solutions of \eqref{eq:ctrsys} to \eqref{eq:lio} through probability measures defined on the space of curves. Toward this end, let $\Gamma : = AC^2([0,1];\Rd)$ be the set of absolutely continuous curves $x$, with $|\dot{x}| \in L^2(0,1;\Rd)$.
	We define the evaluation map $e_t:\Rd \times \Gamma \rightarrow \Rd$ given by,
	\begin{equation}
		e_t(x,\gamma) = \gamma(t) 
	\end{equation}
	for all $t \in [0,1]$. The map $e_t$ takes any measure $\eta$ on $\Rd \times \Gamma$ and defines a measure on $\Rd$ through the pushforward operation, defined by
	\[(e_t)_{\#}\eta(A) = \eta(e_t^{-1}(A))\]
	for all measurable $A \subseteq \Rd$. We define the set $\Gamma_{ad} \subset \Gamma$ by
	\begin{align*}
		\Gamma_{ad}=\lbrace &(x,\gamma) \in \Gamma; \exists u \in L^2(0,1;\mathbb{R}^m), \nonumber 
		\\ & {\rm st.t.}~ \dot{\gamma}(t) = \sum_{i=1}^m u_i(t)g_i(t) ~{\rm and}~~\gamma(0)=x,~~\gamma(t) \in \Omega ~~\forall t \in [0,1],  \rbrace
	\end{align*}
	Let $v:[0,T] \times \Rd \rightarrow \Rd$ be a vector field. We recall the weak notion of solution of the continuity equation. Let $t \mapsto \mu_t$ be a family of probability measures.
	Consider the equation
	\begin{equation}
		\label{eq:nocctyeq}
		\partial_t \mu_t + \nabla_x \cdot (v(t,x)\mu_t) = 0 \hspace{2mm} \text{in} ~~ \Rd \times (0,1)
	\end{equation}
	We will say that $\mu_t$ is a distributional solution to \eqref{eq:nocctyeq} if $\mu_t$ is a solution to \eqref{eq:nocctyeq} if
	\begin{equation}
		\int_0^1 \int_{\Rd} \big (\partial_t \phi(x,t) + v(t,x) \cdot \nabla_x \phi(x,t) \big ) d\mu_t(x)dt
	\end{equation}
	for all $\phi \in C^{\infty}_c(\Rd \times (0,T))$
	
	The following result will be useful in the sequel. The controls we construct will be possibly irregular. Thus negating the possibility of relating solutions of \eqref{eq:ctrsys} and \eqref{eq:lio} using classical results on constructing solution of \eqref{eq:lio} using the flow map of the equation \eqref{eq:ctrsys}.
	
	\begin{theorem}(\textbf{\cite{ambrosio2005gradient} Superposition principle})
		Let $t \mapsto \mu_t$ be a narrowly continuous curve on the space of probability measures on $\Rd$ satisfying the continuity equation \eqref{eq:nocctyeq} satisfying
		\begin{equation}
			\int_0^1\int_{\Rd} |v(t,x)|^2d\mu_t(x) < \infty
		\end{equation}
		Then there exists a probability measure $\eta$ in $\Rd \times \Gamma$ such that $\eta$ is concentrated on the pairs $(x,\gamma) \in \Rd \times \Gamma$ such that $\gamma$ is a solution of the differential equation 
		\begin{equation}
			\dot{\gamma}(t) = v(t,\gamma(t))~~\gamma(0)=x.
		\end{equation}
		and $\mu_t = (e_t)_{\#} \eta$ for all $t \in [0,1]$.
	\end{theorem}
	In the context of control systems, this result can be specialized in the following trivial way.

	\begin{corollary}
		\label{cor}
		Let $t \mapsto \mu_t$ be a narrowly continuous curve on the space of probability measures on $\Rd$ that is supported on $\Omega$ for all $t \in [0,1]$. Additionally, suppose the curve satisfies the continuity  equation \eqref{eq:nocctyeq} for $v(t,x) = \sum^m_i u_i(t,x)g_i(x)$  for a feedback control law $u : [0,1] \times \Rd \rightarrow \R^m$ and
		\begin{equation}
			\int_0^1\int_{\Rd} \sum_{i=1}^m|u_i(t,x)|^2d\mu_t(x) < \infty
		\end{equation}
		Then there exists a probability measure $\eta$ in $\Rd \times \Gamma_{ad}$ such that $\eta$ is concentrated on the pairs $(x,\gamma) \in \Rd \times \Gamma_{ad}$ such that $\gamma$ is a solution of the differential equation 
		\begin{equation}
			\dot{\gamma}(t) = \sum_{i=1}^m u_i(t,\gamma(t))g_i(\gamma(t))~~\gamma(0)=x.
		\end{equation}
		and $\mu_t = (e_t)_{\#} \eta$ for all $t \in [0,1]$.
	\end{corollary}
	
	Lastly, we define some notions of reachability that will be used to prove Theorem \ref{thm:invctrb}.
	
	\begin{definition} 
		A set $A$ is reachable from $B$, if there exists  $x \in B$, $y \in A$ and a control $u \in L^2(0,1;\R^m)$ such that the solution \eqref{eq:ctrsys} satisfies $\gamma(0) = x$ and $\gamma(1) =y$, with $\gamma \in \Omega$.
	\end{definition}

	\begin{definition} \textbf{(Forward Reachable set)}
		The forward reachable set of a set $B \subseteq \Rd$, is ${\rm Reach}(B) :=\{ y \in \Omega; \exists \gamma \in \Gamma_{ad} ~ s.t. \hspace{2mm} \gamma(0) = x, ~ \gamma(1) \in B\} $
	\end{definition}
	
	\begin{definition} \textbf{(Backward reachable set)}
		The backward reachable set of a set $B \subseteq \Rd$, is
		${\rm Reach}^{-1}(B) := \{ x \in \Omega; \exists \gamma \in \Gamma_{ad} ~ s.t. \hspace{2mm} \gamma(0) = x, ~ \gamma(1) \in B \}$
	\end{definition}
	Note that since the system \eqref{eq:ctrsys} doesn't have any drift, ${\rm Reach}^{-1}(B) ={\rm Reach}(B)$ for all sets $B \subseteq  \Omega$.
	\section{Proof of the main result}
	In this section we prove the main result (Theorem \ref{thm:invctrb}).
	
	\subsection{Perron-Frobenius Approach}
	\label{sec:PF}
	We first make the following observation that the amount of mass flowing into a set, by moving along adissible curves, is bounded by the mass present in its backward reachable set. 
	
	\begin{lemma}
		\label{lem}
		Suppose $\eta $ is a probability measure on $\Rd \times \Gamma$ is concentrated on the (closed) set $\Gamma_{ad}$. Suppose additionally, $A,B \subset \Omega$ are measurable sets such that the set $B$ is not reachable from the set $A$. Then $(e_t)_{\#} \eta(B) \leq (e_0)_{\#} \eta(\Omega - A) $ for all $t \in [0,1]$.
	\end{lemma}
	\begin{proof}
		\begin{align*}
			(e_t)_{\#} \eta(B) &= \int_{\Rd}1_B(x) d(e_t)_{\#} \eta(x) \nonumber \\ 
			&= \int_{\Rd \times \Gamma} 1_B(\gamma(t))d\eta(x,\gamma) \nonumber \\  & \leq  \int_{\Omega - A \times \Gamma} 1_B(\gamma(t))d\eta(x,\gamma) \nonumber \\
			&\leq \int_{\Rd\times \Gamma} 1_{A - \Omega}(\gamma(0))d\eta(x,\gamma) \nonumber \\
			&= \int_{\Rd} 1_{A - \Omega}(x)d(e_0)_{\#} \eta(x) \nonumber. \\
			&= (e_0)_{\#} \eta(\Omega - A) \nonumber 
		\end{align*}
	\end{proof}
	
	Next we prove Theorem \ref{thm:ctrbinv} to establish that the control constructed in \eqref{eq:track} provides a solution that solves the continuity equation \eqref{eq:lio} in the distributional sense. This is just reproving the result of \cite{khesin2009nonholonomic} under the weaker setting that we don't assume the flow of the corresponding differential equation \eqref{eq:ctrsys} is well defined.
	
	\vspace{2mm}
	
	\begin{proof}
		Define the $L^2$ interpolation given by 
		
		\[\rho(t,\cdot) = \rho_0 + t(\rho_1 - \rho_0) \hspace{5mm} \forall ~ t \in [0,1] \]
		It's easy to see that $\partial_t \rho(t,\cdot)= \rho_1-\rho_0$ for all $t \in [0,T]$, and $\int_{\Omega }\partial_t\rho(t,x)dx =0$. Since $\Delta_H$ is invertible on $L^2_{\perp}(\Omega)$,
		there exists a unique solution $f = \Delta_H^{-1} (\rho_1 - \rho_0)$. By definition of the operator $\Delta_H$ we know 
		that
		\begin{equation}
			- \int_{\Omega} \partial_t\rho(t,x)\phi(t,x)dx =  -\int_{\Omega} \big(\sum_{i=1}^m X_i f(x)) \big ( \sum_{i=1}^m X_i\phi(t,x) \big ) dx \nonumber
		\end{equation}
		for all $t \in (0,1)$, $\phi \in C_{c}^{\infty}((0,1) \times \Rd)$.
		This implies that 
		\begin{equation}
			\int_{\Omega} \rho(t,x)\partial_t\phi(t,x)dx =  -\int_{\Omega} \big(\sum_{i=1}^m g_i(x)X_i f(x)) \cdot \big ( \nabla_x \phi(t,x) \big ) dx \nonumber
		\end{equation}
		From this we can conclude
		\begin{equation}
			\int_{\Omega} \rho(t,x)\partial_t\phi(t,x)dx =  -\int_{\Omega} \big(\sum_{i=1}^m g_i(x)\frac{X_i f(x)}{\rho(t,x)}\big ) \cdot \big ( \nabla_x \phi(t,x) \big ) \rho(t,x)dxdt \nonumber
		\end{equation}
		This can be rewritten as 
		
		\begin{equation}
			\int_{\Omega} \bigg ( \partial_t\phi(t,x)dx +\int_{\Omega} \big(\sum_{i=1}^m g_i(x)\frac{X_i f(x)}{\rho(t,x)}\big ) \cdot \big ( \nabla_x \phi(t,x) \big ) \bigg) \rho(t,x)dxdt = 0 \nonumber
		\end{equation}
		for all $t \in (0,1)$, $\phi \in C_{c}^{\infty}((0,1) \times \Rd)$.
		And hence, 
		\begin{equation}
			\int_0^1 \int_{\Omega} \bigg ( \partial_t\phi(t,x)dx +\int_{\Omega} \big(\sum_{i=1}^m g_i(x)\frac{X_i f(x)}{\rho(t,x)}\big ) \cdot \big ( \nabla_x \phi(t,x) \big ) \bigg) \rho(t,x)dxdt = 0 \nonumber
		\end{equation}
		$\phi \in C_{c}^{\infty}((0,1) \times \Rd)$.
		By extending $\rho$ trivially to $0$ outside the domain $\Omega$, the integral with respect to $x$ can be taken over $\Rd$ instead of $\Omega$.
	\end{proof}
	
	\vspace{2mm
	}
	Next we are ready to prove our main result.
	
	\vspace{2mm}
	
	{\it Proof of Theorem \ref{thm:invctrb}}
	
	Suppose the restriction of $\Delta_H$ is invertible on $L^2_{\perp}(\Omega)$. Let $K_{\alpha}(r) = e^{-\frac{\|r\|^2}{\delta}}$ for all $ r \in \Rd$.  We define $\rho_0, \rho^{\alpha}_1 \in L^2(\Omega)$ by
	\[\rho_0(p) = \frac{1}{\int_{\Omega}dz} , ~~~ \rho^{\alpha}_1(p)= \frac{1}{\int_{\Omega} K_{\alpha}(z-y)dz }K_{\alpha}(p-y)\]
	for almost every $p \in \Omega$. The function $\rho^{\delta}_1$ weakly converges to the Dirac measure concentrated on $y$, denoted by $\delta_{y}$, as probability measures, as $\alpha \rightarrow 0$.

	Let $u_i(t,x)$ be the controls transporting the solution of the continuity equation from $\rho_0$ to $\rho^{\delta}_1$ as stated in Theorem \ref{thm:ctycontrol}. Since the operator $\Delta^{-1}_H$ defines a bounded inverse on $L^2_{\perp}(\Omega)$, and $\rho^{}_0$ and $\rho^{\delta}_1$ are uniformly bounded from below, we can conclude that $f = \Delta^{-1}_H (\rho^{\delta}_1 - \rho^{}_0)$ is bounded in $WH^1(\Omega)$. This implies the control $u(t,x)$
	satisfies,
	\begin{equation}
		\int_0^1 \int_{\Omega} \sum_{i=1}^m|u_i(t,x)|^2\rho(t,x)dx   <\infty
	\end{equation}
	with $\rho(t, \cdot) = \rho_1 + t (\rho_1 -\rho_0)$ is supported in $\Omega$, for all $t \in [0,T]$. 
	By extending $u(t,x)$ and $\rho(t,x)$ trivially to $0$ outside the domain $\Omega$ we have, 
	\begin{equation}
		\int_0^1 \int_{\Rd} \sum_{i=1}^m|u_i(t,x)|^2\rho(t,x)dx   <\infty
	\end{equation}
	Since $\rho$ is continuous in $L^2(\Omega)$ with respect to time, it is also narrowly continuous (since $t \mapsto \int_{\Rd} \phi(t,x) \rho(t,x)dx$ is continuous for every contiuous function $\phi$).
	It follows then from Corollary \ref{cor}, there exists a probability measure $\eta $ in $\Rd \times \Gamma$ such that $\eta $ is concentrated on solution pairs $(x,\gamma)$ such that $\gamma \in AC^2(0,1;\Rd)$ the set of absolutely continuous curves that are solutions of the equation
	\begin{equation}
		\dot{\gamma}(t) = \sum_{i=1}^m u_i(t,\gamma(t))g_i(\gamma(t))~~\gamma(0)=x.
	\end{equation}
	Fix $\varepsilon>0 $. For every $R>0$, there exists $ \delta >0$ such that $\int_{B_R(y)}\rho^{\delta}_1(x)dx > 1-\varepsilon$. If there exists a set $A \subseteq \Omega$ of non-zero Lebesgue measure such that $B$ is not reachable from $A \subseteq \Rd$, then we would have that $\int_{B_R(y)}\rho^{\delta}dx \leq \int_{\Omega - A}dx $ according to Lemma \ref{lem}. This results in a contradiction.
	\qed
	
	\subsection{Koopman Approach}
	\label{sec:koop}
	In this section, we develop a dual point of view proving essentially the same result as in Theorem \ref{thm:invctrb}. Particularly, we look at how observables evolve along the flow, instead of measure as we did in the previous section.

	Given this result we prove the following result. 
	
	\begin{theorem}
		\label{thm:koop}
		Suppose the restriction of the operator $\Delta_H$ is injective in $L^2_{\perp}(\Omega)$, then ${\rm Reach}(B_R(y))$ has full Lebesgue measure in $\Omega$, for every $R>0$ and every $y \in \Omega$.
	\end{theorem}
	\begin{proof}
		Let $A=  {\rm Reach}(B_R(y))$ for some $R>0$, and some $y \in \Omega$ Additionally, let $B= \Omega - A$. Let $f = \frac{1}{\int_{A}dx} \xi_A - \frac{1}{\int_{B}dx}\xi_B$, where $\xi_A$ and $\xi_B$ denote the characteristic functions of set $A$ and $B$, respectively. 
		
		First, we note that
		\[f(e^{tg_i}(x)) = {\rm constant}\]
		for all $ t$ small enough, for all $x \in \Omega$.
		Thus, we can compute the {\it intrinsic derivative} of $f$ along the flow $e^{t g_i}$ of the vector-field $g_i$ to get,
		\begin{equation}
			\label{eq:observ}
			\frac{d}{dt}f(e^{tg_i}(x))|_{t=0} = 0 
		\end{equation}
		for all $x \in \Omega$. Therefore, $0$ is a candidate for the weak derivative $X_if$ of $f$. But since $f$ is not differentiable in the classical sense, it is not clear if $X_if$ exists as a weak derivative, and whether we can conclude $X_if (x)=\frac{d}{dt}f(e^{tg_i}(x))|_{t=0}$. Toward this end, we follow the idea in \cite{bramanti2022hormander} to relate intrinsic derivatives along vector fields and weak derivatives. We write the expression 
		\begin{equation}
			\label{eq:wkder}
			\int_{\Omega}f(x)X^*_i\phi(x)dx = -\int_{\Omega} f(x)X(x) \phi (x)dx - \int_{\Omega} f(x) \phi(x) \sum_{j=1}^d \partial_{x_j}g_i^j(x)dx
		\end{equation}
		
		Let $\phi \in C^{\infty}_0(\Omega)$. 
		
		According to the computation in proof of \cite[Proposition 2.22]{bramanti2022hormander},
		\begin{eqnarray}
			-\int_{\Omega} f(x)X\phi(x)dx = -\lim_{t \rightarrow 0} \Bigg \{ \int_{\Omega}\frac{1}{t}[f(e^{-tg_i}(x))- f(x)]\phi(y)dy \nonumber \\  - \int_{\Omega} f(e^{-tg_i}(x))\phi(y)J_t(x)dx  \Bigg \}
		\end{eqnarray}
		where $J^i_t$ denotes the Jacobian of the flow $e^{tg_i}$ and $e^{-t g_i} $ denotes the inverse of the flow. The first term is $0$ because of the computation in \eqref{eq:observ}. For the second term, we can apply Lebesgue's theorem as in proof of \cite[Proposition 2.22]{bramanti2022hormander} to get 
		
		\begin{equation}
			-\int_{\Omega} f(x)X\phi(x)dx = \int_{\Omega} f(x) \phi(x) \sum_{j=1}^d \partial_{x_j}g_i^j(x)dx 
		\end{equation}
		Substituting this in \eqref{eq:wkder}
		we get,
		
		\begin{equation}
			\int_{\Omega}f(x)X^*_i\phi(x)dx =0
		\end{equation}
		Comparing this with the definition of the weak derivative \eqref{eq:wkderdef} we can conclude that the weak-derivative $X_if$ exists and is equal to $0$. Therefore $X_if = 0 $ for all $i = 1,...,m$. This implies that $\sigma(f,f)$ is 0, and hence $\Delta_H f = 0$. Thus, $\Delta_H$ cannot be injective unless either $A$ or $B$ have measure $0$. Since $A$ contains $B_R(y)$, this must imply that $\Omega - A = B$ has measure $0$. This concludes the proof.
	\end{proof}
	Note that in Theorem \ref{thm:koop}, we have assumed less on the operator (injectivity) in comparison with the invertibility requirement in \ref{thm:ctrbinv}. One can indeed relax this assumption in Theorem \ref{thm:invctrb}. However, the proof becomes a bit longer, as one has to approximate the functions $\rho_0$ and $\rho^{\alpha}_1$ in the proof using elements in the range of the operator $\Delta_H$. This can be done because injectivity and self-adjointness imply the range of the operator is dense. However, if the operator is injective, but not invertible, then the Poincar\'e inequality \eqref{eq:poin} cannot hold anymore.
	
	\begin{remark}
		\textbf{(The case of $\Omega = \Rd$)}
		One cannot expect the operator $\Delta_H$ to be invertible on $L_2(\Rd)$, as invertibility even fails for the usual Laplacian due to the presence of harmonic functions. In this case, one can instead consider the weighted Laplacian operator,
		\begin{equation}
			\label{eq:nonlap}
			\Delta^a_H h = -\sum_{i=1}^{m}X_i(a(x)X_ih),  
		\end{equation}
		for a suitable weight function $a \in L^1(\Omega) \cap L^{\infty}(\Omega)$. Then the Poincar\'e test becomes
		\begin{equation} 
			\int_{\Omega} |f(x)|^2a(x)dx \leq \frac{1}{\lambda} \int_{\Omega}\sum_{i=1}^m |X_i f (x)|^2a(x)dx 
		\end{equation}
	\end{remark}
	
	\begin{remark}
		\textbf{(The case of compact manifolds without boundary)} The results of this note can easily be extended to the case of manifold without boundary. A similar superposition principle as in Lemma \ref{lem} can be found in \cite{bernard2008young}, proved for flows on manifolds. 
	\end{remark}
	
	\section{Acknowldegements}
	
	The author thanks Rohit Gupta, Matthias Kawski and Emmanuel Tr\'elat for helpful comments and suggestions.

	\bibliographystyle{plain}
	\bibliography{Ctrb}
\end{document}